\documentclass[12pt,leqno]{article}  % 建檔日期 2014年8月8日
\usepackage{amsmath,amssymb,bbm}
\usepackage{amsthm,array,mathdots}
\usepackage{titlesec} % 重設章節標題所需的巨集
\usepackage[all,cmtip]{xy}
\usepackage{blkarray,arydshln} % 矩陣分隔線所需的巨集

\titlespacing{\section}{0cm}{3.5pc}{1.5pc}

\makeatletter
\def\@citex[#1]#2{\if@filesw\immediate\write\@auxout{\string\citation{#2}}\fi
  \def\@citea{}\@cite{\@for\@citeb:=#2\do
    {\@citea\def\@citea{\@citesep}\@ifundefined
       {b@\@citeb}{{\bf ?}\@warning
       {Citation `\@citeb' on page \thepage \space undefined}}%
{\csname b@\@citeb\endcsname}}}{#1}}
\def\@citesep{; }
\makeatother

\newtheoremstyle{Kang}{}{}{\itshape}{}{\bf}{}{.5em}{}
\theoremstyle{Kang}
\newtheorem{theorem}{Theorem}[section]

\newtheorem{lemma}[theorem]{Lemma}

\newtheoremstyle{Kremark}{}{}{}{}{\bf}{}{.5em}{}
\theoremstyle{Kremark}
\newtheorem*{remark}{Remark.}
\newtheorem{defn}[theorem]{Definition}
\newtheorem{example}[theorem]{Example}
\newtheorem{other}{}

\allowdisplaybreaks[1]  % 允許多行數學式分頁

\def\fn#1{\operatorname{#1}} % function work like \sin
\def\bm#1{\mathbbm{#1}}
\def\c#1{\mathcal{#1}}

\title{Invertible Lattices}

\author{Esther Beneish$^{(1)}$ and Ming-chang Kang$^{(2)}$ \\[3mm]
\begin{minipage}{16cm} \begin{description} \itemsep=-1pt
\item[] $^{(1)}$Department of Mathematics, University of
Wisconsin-Parkside, Kenosha, WI, USA, E-mail: beneish@uwp.edu
\item[] $^{(2)}$Department of Mathematics, National Taiwan
University, Taipei,\\ Taiwan, E-mail: kang@math.ntu.edu.tw \item[]
\end{description} \end{minipage}}

\date{}

\begin{document}

\maketitle

\footnote{\textit{\!\!\!$2010$ Mathematics Subject Classification}. 11R33, 20C10, 20J06, 14E08.}
\footnote{\textit{\!\!\!Keywords and phrases}.
Integral representations, algebraic tori, permutation lattices, normal domains, cohomological Mackey functors.}

\begin{abstract}
{\noindent\bf Abstract.} Theorem. Let $\pi$ be a finite group of
order $n$, $R$ be a Dedekind domain satisfying that (i)
$\fn{char}R=0$, (ii) every prime divisor of $n$ is not invertible
in $R$, and (iii) $p$ is unramified in $R$ for any prime divisor
$p$ of $n$. Then all the flabby (resp.\ coflabby) $R\pi$-lattices
are invertible if and only if all the Sylow subgroups of $\pi$ are
cyclic. The above theorem was proved by Endo and Miyata when
$R=\bm{Z}$ \cite[Theorem 1.5]{EM}. As applications of this
theorem, we give a short proof and a partial generalization of a
result of Torrecillas and Weigel \cite[Theorem A]{TW}, which was
proved using cohomological Mackey functors.

\end{abstract}

\newpage
%------------------------------------S1
\section{Introduction}

Let $\pi$ be a finite group, $R$ be a Dedekind domain (i.e.\ a
commutative noetherian integral domain which is integrally closed
with Krull dimension one). Denote by $R\pi$ the group ring of
$\pi$ over $R$. An $R\pi$-lattice $M$ is a finitely generated left
$R\pi$-module which is a torsion-free $R$-module when regarded as
an $R$-module \cite[page 524]{CR}. $R\pi$-lattices play an
important role in the modular representation theory of the group
$\pi$ \cite[Section 18]{CR}. They arose, when $R=\bm{Z}$, in the
study of Noether's problem and in the birational classification of
algebraic tori \cite{Sw1,EM,Vo,CTS}.

Before discussing the main results, we recall some definitions.

\begin{defn} \label{d1.1}
Let $M$ be an $R\pi$-lattice where $R$ is a Dedekind domain and
$\pi$ is a finite group. $M$ is called a permutation lattice if it
is an $R$-free $R\pi$-module with an $R$-free basis permuted by
$\pi$; explicitly, $M=\bigoplus_{1\le i\le m}R\cdot x_i$ and
$\sigma\cdot x_i=x_j$ for all $\sigma\in\pi$, for all $1\le i \le
m$ (note that $j$ depends on $\sigma$ and $i$). An $R\pi$-lattice
$M$ is called an invertible lattice if, as an $R\pi$-modules, it
is a direct summand of some permutation $R\pi$-lattice. An
$R\pi$-lattice $M$ is called a flabby (or flasque) lattice if
$H^{-1}(\pi',M)=0$ for all subgroups $\pi'$ of $\pi$ \cite[Section
8; CTS; Be, page 103]{Sw1} where $H^{-1}(\pi',M)$ denotes the Tate
cohomology \cite[page 102]{Be}. Similarly, $M$ is called a
coflabby (or coflasque) lattice if $H^1(\pi',M)=0$ for all
subgroups $\pi'$ of $\pi$. Clearly, ``permutation" $\Rightarrow$
``invertible" $\Rightarrow$ ``flabby" and ``coflabby" \cite[Lemma
8.4]{Sw1}.
\end{defn}

\begin{defn} \label{d1.2}
Let $p$ be a prime number and $R$ be a Dedekind domain with
$\fn{char}R=0$. We call $p$ is unramified in $R$ if $p$ is not
invertible in $R$ and the principal ideal $pR$ is an intersection
of some maximal ideals of $R$.
\end{defn}

In \cite{EM,CTS}, many interesting results about $\bm{Z}\pi$-lattices were obtained.
Here is one sample of them.

\begin{theorem}[Endo and Miyata {\cite[Theorem 1.5]{EM}}] \label{t1.3}
Let $\pi$ be a finite group,
$I_{\bm{Z}\pi}:=\fn{Ker}\{\varepsilon:\bm{Z}\pi\to \bm{Z}\}$ be
the augmentation ideal of $\bm{Z}\pi$,
$I_{\bm{Z}\pi}^0:=\fn{Hom}_{\bm{Z}}(I_{\bm{Z}\pi},\bm{Z})$ be the
dual $\bm{Z}\pi$-lattice of $I_{\bm{Z}\pi}$. Then the following
statements are equivalents,
\begin{enumerate}
\item[$(1)$] All the flabby (resp.\ coflabby) $\bm{Z}\pi$-lattices
are invertible; \item[$(2)$] $[I_{\bm{Z}\pi}^0]^{fl}$ is
invertible; \item[$(3)$] All the Sylow subgroups of $\pi$ are
cyclic.
\end{enumerate}
(The definition of $[M]^{fl}$ for an $R\pi$-lattice $M$ can be
found in Definition \ref{d2.2}.)

\end{theorem}

One of the main results of this paper is to generalize the above
theorem for $\bm{Z}\pi$-lattices to the case of $R\pi$-lattices
for some ``nice" Dedekind domain $R$. We remark that many results
for $\bm{Z}\pi$-lattices in \cite{EM,CTS} may be extended readily
to the category of $R\pi$-lattices where $R$ is a Dedekind domain
such that $\fn{char}R=0$ and every prime divisor of $|\pi|$ is not
invertible in $R$. However, in some situations, more delicate
conditions of $R$ are required. It is the case for the following
theorem.

\begin{theorem} \label{t1.4}
Let $\pi$ be a finite group of order $n$, $R$ be a Dedekind domain satisfying that {\rm (i)} $\fn{char}R=0$,
{\rm (ii)} every prime divisor of $n$ is not invertible in $R$,
and {\rm (iii)} $p$ is unramified in $R$ for any prime divisor $p$ of $n$.
Then the following statements are equivalent,
\begin{enumerate}
\item[$(1)$]
All the flabby (resp.\ coflabby) $R\pi$-lattices are invertible;
\item[$(2)$]
$[I_{R\pi}^0]^{fl}$ is invertible where $I_{R\pi}=\fn{Ker}\{\varepsilon:R\pi\to R\}$ is the augmentation ideal of $R\pi$,
and $I_{R\pi}^0=\fn{Hom}_R (I_{R\pi},R)$ is the dual lattice of $I_{R\pi}$;
\item[$(3)$]
All the Sylow subgroups of $\pi$ are cyclic.
\end{enumerate}

\end{theorem}

Besides the standard method in \cite{EM,Sw2}, the crux of the
proof of the above theorem is Theorem \ref{t3.3} which provides a
sufficient condition to ensure $R_1 \otimes_{R_0} R_2$ is a normal
domain when $R_0, R_1, R_2$ are normal domains.

The above theorem will break down if the third assumption about
the unramifiedness is waived. We thank Prof. Shizuo Endo who
provides such an example (see Example \ref{ex4.3}).

We give two applications of Theorem \ref{t1.4}. The first
application is a short proof of the following theorem.

\begin{theorem}[Torrecillas and Weigel {\cite[Theorem A and Corollary 6.7]{TW}}] \label{t1.5}
Let $\pi$ be a cyclic $p$-group and $R$ be a DVR such that $\fn{char} R=0$ and $pR$ is the maximal ideal of $R$.
Let $M$ be an $R\pi$-lattice.
Then the following statements are equivalent,
\begin{enumerate}
\item[$(1)$]
$M$ is a permutation $R\pi$-lattice,
\item[$(2)$]
$M$ is a coflabby $R\pi$-lattice,
\item[$(3)$]
$M$ is a flabby $R\pi$-lattice.
\end{enumerate}
\end{theorem}

The second application of Theorem \ref{t1.4} is to determine
$F_{R\pi}$ when $\pi$ is a cyclic group (for $F_{R\pi}$, see
Definition \ref{d2.1}). Consequently, a partial generalization of
Theorem \ref{t1.5} is obtained if $\pi$ is a cyclic group and $R$
is some semilocal ``nice" Dedekind domain (see Theorem
\ref{t5.4}).

We indicate briefly how to deduce Theorem \ref{t1.5} from Theorem
\ref{t1.4}. First rewrite Theorem \ref{t1.5} as follows.

\begin{theorem} \label{1.6}
Let $\pi$, $R$ be the same as in Theorem \ref{t1.5}. Then $(1)$,
$(2)$, $(3)$ and $(4)$ are equivalent where $(4)$ is
\begin{enumerate}
\item[$(4)$]
$M$ is an invertible $R\pi$-lattice.
\end{enumerate}
\end{theorem}

In fact, $(2)\Leftrightarrow (4)$ (resp.\ $(3)\Leftrightarrow
(4)$) follows from Theorem \ref{t1.4}. As to $(1)\Leftrightarrow
(4)$, it follows from the following theorem in \cite{Be}.

\begin{theorem}[Beneish {\cite[Theorem 2.1]{Be}}] \label{t1.7}
Let $\pi$ be a $p$-group and $R$ be a DVR such that $\fn{char} R=0$ and $p$ is not invertible in $R$.
If $M$ is an invertible $R\pi$-lattice,
then it is a permutation $R\pi$-lattice.
\end{theorem}

Note that Theorem \ref{t1.7} is implicit in the proof of \cite[Theorem 3.2]{EM}.

We remark that \cite[Theorem C]{TW} follows also from Theorem
\ref{t1.4}, because we may take a flabby resolution $0 \to M^0 \to
P \to E \to 0$ (see Definition \ref{d2.2}) and apply Theorem
\ref{t1.4} to $E$. Then take the dual of this exact sequence.

This paper is organized as follows. In Section 2 we recall the
definition of flabby resolutions and flabby class monoids. In
Section 3 we prove that $R[X]/\langle \Phi_n(X)\rangle$ is a
Dedekind domain when $R$ is a ``nice" Dedekind domain. The proof
of Theorem \ref{t1.4} is provided in Section 4 following that of
\cite{EM,Sw2}. Section 5 contains a computation of the flabby
class group $F_{R\pi}$ when $\pi$ is a cyclic group, which
generalizes some part of a theorem of Endo and Miyata
\cite[Theorem 3.3; Sw2, Theorem 2.10]{EM}. Then a partial
generalization of Theorem \ref{t1.5} is given; see Theorem
\ref{t5.4}.

Terminology and notations. A commutative noetherian integral
domain $R$ is called a DVR if it is a discrete rank-one valuation
ring. We denote by $R[X]$ the polynomial ring of one variable over
$R$. $\Phi_m(X)$ denotes the $m$-th cyclotomic polynomial, and
$\zeta_n$ denotes a primitive $n$-th root of unity. We denote by
$R\pi$ the group ring of the finite group $\pi$ over the ring $R$.
If $M$ is an $R\pi$-lattice, then $M^0$ denotes its dual lattice,
i.e. $M^0=Hom_R(M,R)$; note that there is a natural action of
$\pi$ on $M^0$ from the left \cite[page 31]{Sw1}. For emphasis, we
remind the reader that the definition that $p$ is unramified in a
Dedekind domain $R$ is given in Definition \ref{d1.2}.

%------------------------------S2
\section{Preliminaries}

From now on till the end of this paper, when we talk about the
group ring $R\pi$, we always assume that $\pi$ is a finite group
of order $n$.

Let $M$ be an $R\pi$-module. The cohomology groups $H^q(G,M)$ and
the homology groups $H_q(G,M)$ can be defined via the derived
functors $\fn{Ext}_{R\pi}^q (R,M)$ and $\fn{Tor}_q^{R\pi} (R,M)$;
the Tate cohomology groups may be defined by the usual way
\cite[page 102]{Be}. When $q\ge 1$, $H^q(G,M)$ may be defined also
by the bar resolution \cite[Chapters 7 and 8;Ev]{Se}.

\medskip
Consider the category of $R\pi$-lattices.
Most results in \cite{EM} and \cite[Section 1]{CTS} remain valid when we replace $\bm{Z}\pi$ by $R\pi$ where $R$ is a Dedekind domain such that $\fn{char}R=0$ and every prime divisor of $|\pi|$ is not invertible in $R$.
In particular, we may define the flabby class monoid $F_{R\pi}$ and the flabby resolution of an $R\pi$-lattice as follows.

\begin{defn}[{\cite[Definition 2.6]{Sw2}}] \label{d2.1}
Let $\pi$ be a finite group of order $n$, $R$ be a Dedekind domain such that $\fn{char}R=0$ and every prime divisor of $n$ is not invertible in $R$.
In the category of flabby $R\pi$-lattices,
we define an equivalence relation ``$\sim$":
Two flabby $R\pi$-lattices $M_1$ and $M_2$ are equivalent, denoted by $M_1\sim M_2$,
if and only if $M_1\oplus P_1\simeq M_2\oplus P_2$ for some permutation lattices $P_1$ and $P_2$.
Let $F_{R\pi}$ be the set of all such equivalence classes.
It is a monoid under direct sum.
$F_{R\pi}$ is called the flabby class monoid of $\pi$.
The equivalence class containing a flabby lattice $M$ is denoted by $[M]$.
\end{defn}

We will say that $[M]$ is invertible (resp.\ permutation) if there is a lattice $E$ such that $M\sim E$ and $E$ is invertible (resp.\ permutation).

\begin{defn} \label{d2.2}
Let $R$ and $\pi$ be the same as in Definition \ref{d2.1}. For any
$R\pi$-lattice $M$, there is an exact sequence of $R\pi$-lattices
$0\to M\to P\to E\to 0$ where $P$ is a permutation $R\pi$-lattice
and $E$ is a flabby $R\pi$-lattice. Such an exact sequence is
called a flabby resolution of $M$ \cite[Lemma 1.1; Sw1, Lemma
8.5]{EM}. If $0\to M\to P'\to E'\to 0$ is another flabby
resolution of $M$, it can be shown that $[E]=[E']$ in $F_{R\pi}$.
We define $[M]^{fl}=[E]\in F_{R\pi}$ (see \cite[Lemma 8.7]{Sw1}).
\end{defn}

\begin{lemma} \label{l2.3}
Let $R$ and $\pi$ be the same as in Definition \ref{d2.1}
\leftmargini=7mm
\begin{enumerate}
\item[$(1)$] {\rm (\cite[Lemma 3.1]{Sw2})} If $0\to M'\to M\to
M''\to 0$ is an exact sequence of $R\pi$-lattices where $M''$ is
an invertible $R\pi$-lattice, then
$[M]^{fl}=[M']^{fl}+[M'']^{fl}$. \item[$(2)$] {\rm (\cite[Lemma
3.3]{Sw2})} If $M$ is an $R\pi$-lattice which is an invertible
lattice over each Sylow subgroup of $\pi$, then $M$ is invertible.
\item[$(3)$] {\rm (\cite[Corollary 2.5]{Sw2})} If $0\to M'\to M\to
M''\to 0$ is an exact sequence of $R\pi$-lattices where $M''$ is
invertible and $M'$ is coflabby, then this exact sequence splits.
Similarly, the exact sequence $0\to M'\to M\to M''\to 0$ splits if
$M'$ is an invertible $R\pi$-lattice and $M''$ is a flabby
$R\pi$-lattice.
\end{enumerate}
\end{lemma}

%---------------------------------------S3
\section{Tensor products of normal domains}

The purpose of this section is to find some sufficient conditions
to ensure that $R[\zeta_n]$ is a Dedekind domain when $R$ is a
Dedekind domain. The problem is reduced to the following: If
$R_0$, $R_1$, $R_2$ are normal domains (i.e.\ commutative
noetherian integral domains which are integrally closed), and $R_0
\subset R_1$, $R_0 \subset R_2$, when is the tensor product
$R_1\otimes_{R_0} R_2$ a normal domain?

We recall two fundamental lemmas.

\begin{lemma}[{\cite[page 172, (42.9)]{Na}}] \label{l3.1}
Let $R$ be  normal domain containing a field $k$ and $K$ be an extension field of $k$.
Suppose that $K$ is separably generated over $k$ and $K\otimes_k R$ is an integral domain.
Then $K\otimes_k R$ is a normal domain.
\end{lemma}

\begin{lemma}[{\cite[page 173, (42.12)]{Na}}] \label{l3.2}
Let $R_1$ and $R_2$ be normal domains containing a DVR which is
designated as $R_0$. Denote by $u$ a prime element of $R_0$.
Assume that {\rm (i)} $R_1\otimes_{R_0} R_2$ is a noetherian
integral domain, {\rm (ii)} both $R_1$ and $R_2$ are separably
generated over $R_0$, and {\rm (iii)} for any prime divisor $Q$ of
$uR_1$, $u \cdot (R_1)_Q=Q \cdot (R_1)_Q$ and $R_1/Q$ is separably
generated over $R_0/uR_0$. Then $R_1\otimes_{R_0} R_2$ is a normal
domain.
\end{lemma}

\begin{remark}
According to \cite[page 146]{Na}, if $R_0$ is a subring of a
commutative integral domain $R$ with $k$, $K$ being the quotient
fields of $R_0$, $R$ respectively, we say that $R$ is separably
generated over $R_0$, if (i) $\fn{char}R=0$, or (ii)
$\fn{char}R=p>0$ and $K\otimes_k k^{1/p}$ is an integral domain.
Consequently, if $\fn{char}k=0$ or $k=\bm{F}_q$ is a finite field,
then $R$ is separably generated over $R_0$. Note that Lemma
\ref{l3.1} and Lemma \ref{l3.2} are due to Nakai and Nagata
respectively; see \cite[page 220]{Na}.
\end{remark}

\begin{theorem} \label{t3.3}
Let $n$ be a positive integer and $R$ be a Dedekind domain.
Denote by $R[X]$ the polynomial ring over $R$.
Assume that {\rm (i)} $\fn{char}R=0$, {\rm (ii)} every prime divisor of $n$ is not invertible in $R$,
and {\rm (iii)} $p$ is unramified in $R$ for any prime divisor $p$ of $n$.
Then $R[X]/\langle \Phi_n(X)\rangle$ is a Dedekind domain and $R[X]/\langle \Phi_n(X)\rangle \simeq R[\zeta_n]$.
\end{theorem}

\begin{proof}
Step 1. Let $\Omega$ be an algebraically closed field containing
$K$ where $K$ is the quotient field of $R$. Let $\zeta_n$ be a
primitive $n$-th root of unity in $\Omega$. We will show that,
within $\Omega$, the subfields $K$ and $\bm{Q}(\zeta_n)$ are
linearly disjoint over $\bm{Q}$.

Let $k=K\cap \bm{Q}(\zeta_n)$.  We will show that $k=\bm{Q}$.

Otherwise, $\bm{Q}\subsetneq k$. Then there is some prime number
$p$ such that $p$ ramifies in $k$. Since $\bm{Q}\subset k\subset
\bm{Q}(\zeta_n)$, it is necessary that $p$ divides $n$. By
assumptions, $p$ is unramified in $R$. Thus $p$ is also unramified
in $k$ because $\bm{Q}\subset k\subset K$. This is a
contradiction.

Once we know $\bm{Q}=k$, it is easy to see that $\Phi_n(X)$ is
irreducible in $K[X]$. Suppose not. Write $\Phi_n(X)=f_1(X)\cdot
f_2(X)$ where $f_1(X), f_2(X)\in K[X]$ are monic polynomials and
$\deg f_1(x)<\deg \Phi_n(X)$. Since the roots of $f_1(X)$ are
primitive $n$-the roots of unity, it follows that the coefficients
of $f_1(X)$ belong to $\bm{Q}(\zeta_n)$. Thus these coefficients
lie in $K\cap \bm{Q}(\zeta_n)=\bm{Q}$. Hence $f_1(X)\in\bm{Q}[X]$,
which is impossible.

We conclude that $\Phi_n(X)$ is irreducible in $K[X]$.
Thus $\bm{Q}(\zeta_n)$ is linearly disjoint from $K$ over $\bm{Q}$ \cite[page 49]{La}.

It follows that the canonical map $\bm{Q}(\zeta_n)\otimes_{\bm{Q}} K\to \bm{Q}(\zeta_n)\cdot K=K(\zeta_n)$ is an isomorphism \cite[page 49]{La}.
Thus $\bm{Z}[\zeta_n]\otimes_{\bm{Z}} R\to R[\zeta_n]$ is also an isomorphism.
Note that $\bm{Z}[\zeta_n]\otimes_{\bm{Z}} R\simeq R[X]/\langle \Phi_n(X)\rangle$.
In particular, $R[X]/\langle \Phi_n(X)\rangle$ is a noetherian integral domain.

\bigskip
Step 2. It remains to show that $R[X]/\langle \Phi_n(X)\rangle$ is
integrally closed. Remember that $R[X]/\langle \Phi_n(X)\rangle
\simeq \bm{Z}[\zeta_n]\otimes_{\bm{Z}} R$.

For any non-zero prime ideal $Q$ of $R$, let $R_Q$ be the
localization of $R$ at $Q$. We will show that $R_Q[X]/\langle
\Phi_n(X)\rangle$ is integrally closed for all such $Q$. Because
$R[X]/\langle \Phi_n(X)\rangle =\bigcap_Q R_Q[X]/\langle
\Phi_n(X)\rangle$, this will show that $R[X]/\langle
\Phi_n(X)\rangle$ is integrally closed.

Suppose that $Q\cap \bm{Z} \neq 0$ and $Q\cap \bm{Z}=\langle
q\rangle$ for some prime number $q$. If $q$ is a divisor of $n$,
then $q$ is unramified in $R$. Let $S=\bm{Z}\backslash \langle
q\rangle$ and $\bm{Z}_q=S^{-1}\bm{Z}$ be the localization of
$\bm{Z}$ at $\langle q\rangle$. Then $S^{-1}R[X]/\langle
\Phi_n(X)\rangle \simeq \bm{Z}_q[\zeta_n]\otimes_{\bm{Z}_q}
(S^{-1}R)$. Apply Lemma \ref{l3.2}. Note that the assumptions of
Lemma \ref{l3.2} are fulfilled, e.g.\ if $Q'$ is a prime divisor
of $qR$, then $S^{-1}R/S^{-1}Q'$ is separably generated over
$\bm{Z}_q/q\bm{Z}_q$ because $\bm{Z}_q/q\bm{Z}_q\simeq \bm{F}_q$
is a finite field. Hence $S^{-1}R[X]/\langle \Phi_n(X)\rangle$ is
a normal domain. Since $R_Q[X]/\langle \Phi_n(X)\rangle$ is a
localization of $S^{-1}R[X]/\langle\Phi_n(X)\rangle$, it follows
that $R_Q[X]/\langle \Phi_n(X)\rangle$ is integrally closed.

Suppose that $Q\cap \bm{Z} \neq 0$ and $Q\cap \bm{Z}=\langle
q\rangle$ for some prime number $q$ such that $q$ is not a divisor
of $n$. Then $q$ is unramified in $\bm{Z}[\zeta_n]$. Thus we my
apply the same arguments as above and apply Lemma \ref{l3.2} to
$S^{-1}\bm{Z}[\zeta_n]=\bm{Z}_q[\zeta_n]$. Hence
$R_Q[X]/\langle\Phi_n(X)\rangle$ is also integrally closed.

Suppose that $Q\cap \bm{Z}=0$. Then $\bm{Q} \subset R_Q$. Let
$T=\bm{Z}\backslash \{0\}$. Then $T^{-1}R[X]/\langle
\Phi_n(X)\rangle\allowbreak\simeq
T^{-1}\bm{Z}[\zeta_n]\otimes_{\bm{Q}}
R_Q=\bm{Q}(\zeta_n)\otimes_{\bm{Q}} R_Q$. Apply Lemma \ref{l3.1}.
We find that $T^{-1}R[X]/\langle \Phi_n(X)\rangle$ is a normal
domain. Hence the result.
\end{proof}

\begin{remark}
We thank Nick Ramsey for pointing out that Proposition 17 of
\cite[page 19]{Se} provides a special case of Theorem \ref{t3.3} :
If $p$ is a prime number and $R$ is a DVR with maximal ideal $pR$,
then $R[X]/\langle\Phi_{p^t}(X)\rangle$ is again a DVR where $t$
is any positive integer. \end{remark}

%---------------------------------------------S4
\section{Proof of Theorem \ref{t1.4}}

The following lemma is a generalization of \cite[Lemma 4.3]{Sw2}.

\begin{lemma} \label{l4.1}
Let $\pi$ be a cyclic $p$-group of order $n$. Write
$\pi=\langle\sigma\rangle$. Let $R$ be a Dedekind domain such that
$\fn{char} R=0$ and $p$ is unramified in $R$. Let $M$ be a
finitely generated module over $R\pi/\langle\Phi_n(\sigma)\rangle$
such that $M$ is a torsion-free $R$-module when it is regarded as
an $R$-module. Then $[M]^{fl}$ is an invertible $R\pi$-lattice.
\end{lemma}

\begin{proof}
Write $n=pq$ and define $\pi''=\pi/\langle\sigma^q\rangle$.
From the factorization $X^n-1= (X^q-1)\Phi_n(X)$, we get an exact sequence $0\to R\pi/\langle\Phi_n(\sigma)\rangle\to R\pi\to R\pi''\to 0$.
Note that $R\pi/\langle\Phi_n(\sigma)\rangle \simeq R[X]/\langle\Phi_n(X)\rangle \simeq R[\zeta_n]$ is a Dedekind domain by Theorem \ref{t3.3}.
This provides a flabby resolution of the $R\pi$-lattice $R\pi/\langle\Phi_n(\sigma)\rangle$.
Hence $[R\pi/\langle\Phi_n(\sigma)\rangle]^{fl}\allowbreak=[R\pi'']=0$.

Since $M$ is torsion-free, $M$ is a projective module over the Dedekind domain $R\pi/\langle\Phi_n(\sigma)\rangle$.
Thus we may find another module $N$ satisfying that $M\oplus N\simeq (R\pi/\allowbreak\langle\Phi_n(\sigma)\rangle)^{(t)}$ for some integer $t$.
Thus $[M]^{fl}+[N]^{fl}=t[R\pi/\langle\Phi_n(\sigma)\rangle]^{fl}=0$.
It follows that $[M]^{fl}$ is invertible.
\end{proof}

\bigskip
\begin{proof}[Proof of Theorem \ref{t1.4}] ~ \par
The proof of Theorem \ref{t1.4} is almost the same as that in
\cite[Theorem 1.3; Sw2, Theorem 4.4]{EM}, once Lemma \ref{l4.1} is
obtained. In order not to commit a blunder mistake, we choose to
rewrite the proof once again.

\bigskip
$(3)\Rightarrow (1)$

Step 1. Assume that all the Sylow subgroups of $\pi$ are cyclic.
Let $M$ be an $R\pi$-lattice which is flabby (resp.\ coflabby). We
will show that $M$ is invertible.

By Lemma \ref{l2.3}, it suffices to show that $M$ is an invertible $R\pi_p$-lattice where $\pi_p$ is a $p$-Sylow subgroup of $\pi$
and $p$ is a prime divisor of $|\pi|$.
Thus we may assume that $\pi=\langle\sigma\rangle$ is a cyclic $p$-group of order $n$,
without loss of generality.

\medskip
Step 2. For any $R\pi$-lattice $M$,
we claim that $M$ is flabby if and only if it is coflabby.

Since $\pi=\langle\sigma\rangle$ is cyclic of order $n$, we find
that $H^{-1}(\pi,M)\simeq \fn{Ker}\varphi/\langle\sigma v-v:v\in
M\rangle$ where $\varphi:M\to M$ is defined by
$\varphi(u)=u+\sigma\cdot u+\cdots+\sigma^{n-1}\cdot u$. On the
other hand, $H^{1}(\pi,M)=\fn{Ker}\varphi/\langle\sigma\cdot
v-v:v\in M\rangle$ by definition. Hence $H^1(\pi,M)\simeq
H^{-1}(\pi,M)$. Similarly, for any subgroup $\pi' \subset \pi$,
$H^1(\pi',M)\simeq H^{-1}(\pi',M)$. Hence the result.

\medskip
Step 3. Let $M$ be a flabby $R\pi$-lattice.
We will show that $M$ is invertible.

Write $n=pq$ where $q$ is a power of $p$.
Define $M'=\{u\in M:\Phi_n(\sigma)\cdot u=0\}$, $M''=M/M'$.
Then we have an exact sequence of $R\pi$-lattices $0\to M'\to M\to M''\to 0$ where $M'$ is a module over $R\pi/\langle\Phi_n(\sigma)\rangle$
and $M''$ is a lattice over $R\pi''$ with $\pi''=\pi/\langle\sigma^q\rangle$.

By Theorem \ref{t3.3}, $R\pi/\langle\Phi_n(\sigma)\rangle \simeq
R[\zeta_n]$ is a Dedekind domain. Thus $[M']^{fl}$ is invertible
by Lemma \ref{l4.1}.

We will show that $M''$ is a flabby $R\pi''$-lattice.
This will be proved in the next step.

Assume the above claim.
By induction on $|\pi|$, we find that $M''$ is invertible.
Thus $[M'']^{fl}$ is invertible also.
Apply Lemma \ref{l2.3}.
We find that $[M]^{fl}=[M']^{fl}+[M'']^{fl}$ is invertible.

Since $[M]^{fl}$ is invertible, we get a flabby resolution of $M$,
$0\to M\to P\to E\to 0$ where $P$ is permutation and $E$ is invertible.
By Step 1, $M$ is coflabby.
Hence the exact sequence $0\to M\to P\to E\to 0$ splits by Lemma \ref{l2.3}.
We get $P\simeq M\oplus E$.
Thus $M$ is invertible.

\medskip
Step 4.
We will show that $M''$ is a flabby $R\pi''$-lattice.

For any subgroup $\pi'$ of $\pi$, we will show that $H^{-1}(\pi',M'')=0$.

If $\pi'=\{1\}$, it is clear that $H^{-1}(\pi',M'')=0$.

Now assume $\pi'\supsetneq \{1\}$. Write
$\pi'=\langle\sigma^d\rangle$ with $d\mid n$ and $d\ne n$. Then
${M'}^{\pi'}:=\{u\in M':\lambda\cdot u=u$ for any $\lambda\in
\pi'\}=0$ because, for any $v\in {M'}^{\pi'}$,
$\Phi_n(\sigma)\cdot v=0$, $(\sigma^d-1)\cdot v=0$, and $M'$ is
torsion-free. From the exact sequence $0=H^{-1}(\pi',M) \to
H^{-1}(\pi',M'') \to \hat{H}^0(\pi',M') =0$, we find
$H^{-1}(\pi',M'')=0$.

Now that $M''$ is flabby as an $R\pi$-lattice, it is flabby as an
$R\pi''$-lattice (where $\pi''=\pi/\langle\sigma^q\rangle$)
because every subgroup of $\pi''$ may be written as
$\pi_1/\langle\sigma^q\rangle$ for some subgroup
$\langle\sigma^q\rangle \subset \pi_1$ and
$H^1(\pi_1/\langle\sigma^q\rangle,M'')\to H^1(\pi_1,M'')$ is
injective by the five-term exact sequence of the
Hochschild-Serre's spectral sequence. Done.

\bigskip
$(1)\Rightarrow (2)$ In general, $[I_{R\pi}^0]^{fl}$ is flabby.
By (1), it is invertible.

\bigskip
$(2)\Rightarrow (3)$ Let $\pi$ be a group of order $n$.
Let $I_{R\pi}$ be the augmentation ideal.
Then we have an exact sequence $0\to I_{R\pi}\to R\pi\to R\to 0$.
Thus $H^1(\pi,I_{R\pi})=R/nR$.

If $[I_{R\pi}^0]^{fl}$ is invertible, then we have an exact
sequence $0\to I_{R\pi}^0\to P\to E\to 0$ where $P$ is permutation
and $E$ is invertible. Taking the dual of each lattice, we get
$0\to E^0\to P^0\to I_{R\pi}\to 0$. Note that $P^0$ is also
permutation and $E^0$ is invertible. Moreover, we have
$0=H^1(\pi,P^0)\to H^1(\pi,I_{R\pi})\to H^2(\pi,E^0)$. Thus there
is an embedding (i.e. an injective map of $R$-modules) $0\to
R/nR\to H^2(\pi,E^0)$.

Write $E^0\oplus E'=Q$ where $Q$ is some permutative $R\pi$-lattice.
It follows that there is also an embedding $0\to R/nR\to H^2(\pi,Q)$.

Write $Q=\bigoplus_i R\pi/\pi_i$ where $\pi_i$'s are subgroups of $\pi$.
Then $H^2(\pi,Q)=\bigoplus_i H^2(\pi,\allowbreak R\pi/\pi_i)\simeq \bigoplus_i H^2(\pi_i,R)$.

Since $p$ is unramified in $R$, choose a prime ideal $P$
containing $pR$. Let $R_P$ be the localization of $R$ at $P$.
Consider $R_P \otimes Q$. In other words, we may assume that $R$
is a DVR with maximal ideal $pR$. We will show that all the Sylow
subgroups of $\pi$ are cyclic. Let $p$ be a prime divisor of $n$
($:=|\pi|$). Write $n=p^tn'$ with $t\ge 1$ and $p\nmid n'$.

 Since $0\to R/nR\to H^2(\pi,Q)\simeq \bigoplus_i
H^2(\pi_i,R)$ and $R/nR\simeq R/p^tR\oplus R/n'R$, it follows that
there is an embedding $0\to R/p^tR\to H^2(\pi_i,R)$ for some $i$,
because $R/p^tR$ is an indecomposable $R$-module.

Thus the proof is finished by the following lemma.
\end{proof}

\begin{lemma}[{\cite[Lemma 4.5]{Sw2}}] \label{l4.2}
Let $\pi$ be a finite group of order $n$, $p$ be a prime divisor
of $n$. Let $R$ be a DVR such that $\fn{char}R=0$ and the maximal
ideal of $R$ is $pR$. Write $n=p^tn'$ where $t\ge 1$ and $p\nmid
n'$. If there is an embedding $0\to R/p^tR\to H^2(\pi,R)$, then
the $p$-Sylow subgroup of $\pi$ is cyclic of order $p^t$.
\end{lemma}

\begin{proof}
From the exact sequence $0\to R \xrightarrow{p^t} R \to R/p^tR\to
0$ of $R\pi$-modules, we get $0=H^1(\pi,R)\to H^1(\pi,R/p^tR)\to
H^2(\pi,R) \xrightarrow{p^t} H^2(\pi,R)$, it follows that there is
an embedding $0\to R/p^tR\to H^1(\pi,R/p^tR)$. Since
$H^1(\pi,R/p^tR)\simeq \fn{Hom}(\pi,R/p^tR)$, there is a group
homomorphism $f:\pi\to R/p^tR$ such that the annihilator
$\fn{Ann}_R (f)=p^tR$ (here $\fn{Hom}(\pi,R/p^tR)$ is regarded as
an $R$-module). Hence $\pi$ contains an element of order $p^t$.
\end{proof}

\bigskip
Prof. Shizuo Endo kindly communicated with the following example
which showed that the assumption of unramifiedness in Theorem
\ref{t1.4} is crucial.

\begin{example} \label{ex4.3}
Let $p$ be an odd prime number and $R=\bm{Z}[\zeta_p]$. Write
$\zeta=\zeta_p$. Let $\pi= \langle \sigma \rangle \simeq C_p$ be
the cyclic group of order $p$. Then $p$ is ramified in $R$; in
fact, $pR= (1- \zeta)^{p-1}$.

Let $M = R \cdot u$ be the cyclic $R \pi$-lattice defined by
$\sigma \cdot u = \zeta u$. Taking a flabby resolution of $M^0$
and then taking the dual, we obtain an exact sequence of $R
\pi$-lattices $0 \to E \to P \to M \to 0$ where $P$ is a
permutation lattice and $E$ is a coflabby lattice (and also a
flabby lattice by the periodicity of cohomolgy groups). We will
show that $E$ is not an invertible lattice. This shows that
Theorem \ref{t1.4} fails for $R\pi$.

It is easy to show that $H^{-1}(\pi, M)= R/(1-\zeta)$, and
$\hat{H}^0(\pi, Q)=(R/(1-\zeta)^{p-1})^{(n)}$ if $Q=R^{(n)} \oplus
(R\pi)^{(n')}$ is a permutation lattice.

Suppose that $E$ is an invertible lattice, then $\hat{H}^0(\pi,
E)$ is a direct summand of $(R/(1-\zeta)^{p-1})^{(m)}$ for some
integer $m$. Since the module $(R/(1-\zeta)^{p-1})^{(m)}$
satisfies the ascending chain condition and the descending chain
condition, the Krull-Schmidt-Azumaya Theorem may be applied to it
\cite[page 128, Theorem 6.12]{CR}. Hence $\hat{H}^0(\pi, E) \simeq
(R/(1-\zeta)^{p-1})^{(m')}$ for some integer $m'$.

From the exact sequence $0 \to E \to P \to M \to 0$, we get an
exact sequence of $R$-modules $0 \to H^{-1}(\pi, M) \to
\hat{H}^0(\pi, E) \to \hat{H}^0(\pi, P) \to 0$, i.e. an exact
sequence $0 \to R/(1-\zeta) \to (R/(1-\zeta)^{p-1})^{(m')} \to
(R/(1-\zeta)^{p-1})^{(n)} \to 0$. Counting the lengths of these
modules, we find a contradiction.

\medskip
For the case $p=2$, let $R=\bm{Z}[\sqrt{-1}]$ and $\pi= \langle
\sigma \rangle \simeq C_2$ be the cyclic group of order $2$. Let
$M = R \cdot u$ be the cyclic $R \pi$-lattice defined by $\sigma
\cdot u = - u$. We can find a flabby $R\pi$-lattice $E$ which is
not invertible as before.

\medskip
More generally, let $\pi= \langle \sigma \rangle \simeq C_n$ be
the cyclic group of order $n$ and $p$ be a prime divisor of $n$.
Suppose that $R$ is a Dedekind domain such that {\rm (i)}
$\fn{char}R=0$, {\rm (ii)} every prime divisor of $n$ is not
invertible in $R$, and {\rm (iii)} $\zeta_p \in R$ if $p$ is odd
(resp. $\sqrt{-1} \in R$ if $p=2$). Define $\pi^"$ to be the
quotient group of $\pi$ with $| \pi" | =p$. Find a flabby, but not
invertible $R\pi"$-lattice $E$ as above. As an $R\pi$-lattice, $E$
is flabby and is not invertible (see, for example, \cite[page 180,
Lemma 2]{CTS}).

\end{example}

\bigskip
For the convenience of the reader, we reproduce a proof of Theorem \ref{t1.7}.

\begin{proof}[Proof of Theorem \ref{t1.7}] ~ \par
Let $M$ be an invertible $R\pi$-lattice.
We will show that $M$ is permutation.

Write $M\oplus M'=P$ for some permutation $R\pi$-lattice and some $M'$.

Denote by $\widehat{R}$ the completion of $R$ at its maximal ideal.
In the category of $\widehat{R}\pi$-lattices,
$\widehat{R}\pi/\pi'$ is indecomposable for any subgroup $\pi'$ of $\pi$ by \cite[page 678, Theorem 32.14]{CR}
(note that the assumptions of \cite[Theorem 32.11]{CR} are satisfied).

Since the Krull-Schmidt-Azumaya Theorem is valid in the category
of $\widehat{R}\pi$-lattices \cite[page 128, Theorem 6.12]{CR},
from $\widehat{R}M\oplus \widehat{R}M'=\widehat{R}P$, we find a
permutation $R\pi$-lattice $Q$ such that $\widehat{R}M\simeq
\widehat{R}Q$. By \cite[page 627, Proposition (30.17)]{CR} we find
that $M\simeq Q$.
\end{proof}

%---------------------------------------------S5
\section{The flabby class group}

Let $R$ be a Dedekind domain. Recall that the class group of $R$,
denoted by $C(R)$, is defined as $C(R)=I(R)/P(R)$ where $I(R)$ is
the group of fractional ideals of $R$ and $P(R)$ is the group of
principal ideals of $R$. If $J$ is a fractional ideal of $R$,
$[J]$ denotes the image of $J$ in $C(R)$. The group operation in
$C(R)$ is written multiplicatively.

Let $R$ be a Dedekind domain and $M$ be a finitely generated torsion-free $R$-module.
Then $M\simeq R^{(m-1)}\oplus I$ where $I$ is a non-zero ideal of $R$.
Define the Steinitz class of $M$, denoted by $\fn{cl}(M)$, by $\fn{cl}(M)=[I]$ (see \cite[page 85]{CR}).
If $M_1$ and $M_2$ are finitely generated torsion-free $R$-modules,
it is not difficult to verify that $\fn{cl}(M\oplus N)=\fn{cl}(M)\cdot \fn{cl}(N)$.

\begin{defn} \label{d5.1}
Let $R$ be a Dedekind domain, $M$ be a finitely generated
$R$-module. Define $M_0=M/\{\mbox{torsion elements in }M\}$.
\end{defn}

\begin{theorem} \label{t5.2}
Let $\pi=\langle\sigma\rangle$ be a cyclic group of order $n$ and
$R$ be a Dedekind domain satisfying that {\rm (i)} $\fn{char}R=0$,
{\rm (ii)} every prime divisor of $n$ is not invertible in $R$,
and {\rm (iii)} $p$ is unramified in $R$ for any prime divisor $p$
of $n$. Define a map $c_\theta:F_{R\pi}\to \bigoplus_{d\mid n}
C(R\pi/\langle\Phi_n(\sigma)\rangle)$ by
$c_\theta([M])=(\ldots,\fn{cl}((M/\Phi_d(\sigma)M)_0),\ldots)$
where $M$ is a flabby $R\pi$-lattice. Then $c_\theta$ is an
isomorphism.
\end{theorem}

\begin{remark}
By Theorem \ref{t1.4}, $F_{R\pi}$ is a group.
By Theorem \ref{t3.3}, $R\pi/\langle\Phi_d(\sigma)\rangle \simeq R[X]/\allowbreak \langle\Phi_d(X)\rangle\simeq R[\zeta_d]$ is a Dedekind domain;
thus $C(R\pi/\langle\Phi_d(\sigma)\rangle)$ is well-defined.
\end{remark}

The proof of Theorem \ref{t5.2} follows by the same way as in
\cite[Sections 5 and 6]{Sw2}. Before giving the proof of it, we
recall a key lemma in \cite{Sw2}.

\begin{theorem} \label{t5.3}
Let $\pi$ and $R$ be the same as in Theorem \ref{t5.2}.
If $M$ is an invertible $R\pi$-lattice, then
\begin{align*}
[M]^{fl} &= \sum_{d\mid n}[(M/\Phi_d(\sigma)M)_0]^{fl}, \\
[(M/\Phi_n(\sigma)M)_0]^{fl} &= \sum_{d\mid n} \mu\left(\frac{n}{d}\right)[M/(\sigma^d-1)M]^{fl}.
\end{align*}
\end{theorem}

\begin{proof}
The proof of Theorem 5.1 and Corollary 5.2 in \cite[Section 5]{Sw2} works as well in the present situation.
The details are omitted.
\end{proof}

\bigskip
\begin{proof}[Proof of Theorem \ref{t5.2}] ~ \par
First of all, we will show that $c_\theta$ is injective. Let $M$
be a flabby $R\pi$-lattice. It is invertible by Theorem
\ref{t1.4}. If $c_\theta([M])=0$ in $\bigoplus_{d\mid
n}C(R\pi/\langle\Phi_d(\sigma)\rangle)$, then
$(M/\Phi_d(\sigma)M)_0$ is a free module over $R\pi/\langle
\Phi_d(\sigma)\rangle$ for all $d\mid n$. Since $[R\pi/\langle
\Phi_d(\sigma)\rangle]^{fl}=0$ (by applying Theorem \ref{t5.3}
with $M=R\pi$), we find $[(M/\Phi_d(\sigma)M)_0]^{fl}=0$. By
Theorem \ref{t5.3}, we find $[M]^{fl}=0$. Thus we have a flabby
resolution of $M$, $0\to M\to P_1\to P_2\to 0$ where $P_1$ and
$P_2$ are permutation $R\pi$-lattices. Since $M$ is invertible, we
may apply Lemma \ref{l2.3} to conclude that $P_1\simeq M\oplus
P_2$. Thus $M\sim P_1$ and $[M]=0$ in $F_{R\pi}$.

It remains to show that $c_\theta$ is surjective.
We also follow the proof of \cite[page 247--248]{Sw2}.

\medskip
Step 1. Let $K_0(R\pi)$ be the Grothendieck group of the category
of finitely generated projective $R\pi$-module. Every such
projective module is isomorphic to a direct sum of a free module
and a projective ideal $\c{A}$ \cite[Theorem A]{Sw3}. Define
$C(R\pi)$ as a subgroup of $K_0(R\pi)$ by
$C(R\pi)=\{[\c{A}]-[R\pi]\in K_0(R\pi):\c{A}$ is a projective
ideal over $R\pi\}$. The group $C(R\pi)$ is called the locally
free class group of $R\pi$ \cite[page 659; EM, page 86]{CR}.

\medskip
Step 2.
Define a map $c':C(R\pi)\to F_{R\pi}$ by $c'([\c{A}]-[R\pi])=[\c{A}]\in F_{R\pi}$.
Since $\c{A}$ is a projective ideal over $R\pi$,
it is an invertible $R\pi$-lattice; thus $c'$ is well-defined.

We will show that the composition map $c_\theta \circ c': %
C(R\pi)\to F_{R\pi}\to \bigoplus_{d\mid n} C(R\pi/\allowbreak \langle\Phi_d(\sigma)\rangle)$
is surjective in the next step.
Once it is proved, $c_\theta$ is also surjective.

\medskip
Step 3.
We will show that $c_\theta\circ c':C(R\pi)\to F_{R\pi}\to \bigoplus_{d\mid n} C(R\pi/\langle\Phi_d(\sigma)\rangle)$ is surjective.

Let $K$ be the quotient field of $R$. Write
$\Omega_{R\pi}:=\prod_{d\mid n} R\pi/\langle
\Phi_d(\sigma)\rangle$. It is not difficult to verify that
$\Omega_{R\pi}$ is the maximal $R$-order in $K\pi$ containing
$R\pi$ \cite[page 559 and page 563]{CR}. We may define the locally
free class group $C(\Omega_{R\pi})$ as in the case $C(R\pi)$ (see
\cite[page 659]{CR}). It follows that $C(\Omega_{R\pi})\simeq
\bigoplus_{d\mid n} C(R\pi/\langle \Phi_d(\sigma)\rangle)$.

The composite map $c_\theta\circ c'$ turns out to be $c_\theta\circ c'([\c{A}]-[R\pi])= %
[\Omega_{R\pi}\otimes_{R\pi} \c{A}]-[\Omega_{R\pi}]\allowbreak \in
C(\Omega_{R\pi})$, which is just the natural map $C(R\pi)\to
C(\Omega_{R\pi})$ (the map induced by the inclusion map $R\pi \to
\Omega_{R\pi}$). Thus the surjectivity of $c_\theta\circ c'$ is
equivalent to the surjectivity of the map $C(R\pi)\to
C(\Omega_{R\pi})$. However, the map $C(R\pi)\to C(\Omega_{R\pi})$
is surjective by \cite[Corollary 11]{Ri}; in applying Rim's
Theorem, we should verify the fact that $R\pi$ has no nilpotent
ideal, which may be see from the embedding $R\pi\hookrightarrow
K\pi \simeq \prod_{d\mid n} K(\zeta_d)$ and hence $R\pi$ has no
nilpotent element. This finishes the proof that $c_\theta\circ c'$
is surjective.

Alternatively, the reader may show that $C(R\pi)\to C(\Omega_{R\pi})$ is surjective by modifying the proof of \cite[Lemma 6.1]{Sw2}.
\end{proof}

Now we give a partial generalization of Theorem \ref{t1.5}.

\begin{theorem} \label{t5.4}
Let $\pi$ be a cyclic group of order $n$, $R$ be a semilocal
Dedekind domain satisfying {\rm (i)} $\fn{char}R=0$, {\rm (ii)}
every prime divisor of $n$ is not invertible in $R$, and {\rm
(iii)} $p$ is unramified in $R$ for every prime divisor $p$ of
$n$. Then $F_{R\pi}=\{0\}$ and all the flabby $R\pi$-lattices are
stably permutation, i.e. if $M$ is a flabby $R\pi$-lattice, there
are permutation $R\pi$-lattices $P_1$ and $P_2$ such that $M\oplus
P_1 \simeq P_2$.
\end{theorem}

\begin{proof}
Apply Theorem \ref{t5.2}.
It suffices to show that $C(R\pi/\langle\Phi_d(\sigma)\rangle)=0$ where $\pi=\langle \sigma \rangle$ is of order $n$ and $d\mid n$.
Note that $R\pi/\langle \Phi_d(\sigma)\rangle \simeq R[\zeta_d]$ is a Dedekind domain integral over $R$.
Since $R$ is semilocal, $R[\zeta_d]$ is also semilocal.
Thus $R[\zeta_d]$ is a principal ideal domain and $C(R[\zeta_d])=0$.
Thus $F_{R\pi}=\{0\}$.

If $M$ is a flabby $R\pi$-lattice, from $[M]\in F_{R\pi}=\{0\}$,
we find that $[M]=0$, i.e.\ $M\sim 0$ which is equivalent to that
$M$ is stably permutation.
\end{proof}

\newpage
%----------------------------------------References

\end{document}